\documentclass[11pt,leqno]{article}
\usepackage{amsmath,amssymb,amsfonts,amsthm,mathrsfs,verbatim} 
\usepackage{graphics} 
\usepackage{eso-pic}
\usepackage{lipsum}
\usepackage{authblk}
\usepackage{graphicx}  
\usepackage[colorlinks=true]{hyperref}
\usepackage{geometry}
\hypersetup{urlcolor=blue, citecolor=red}

\newtheorem{theorem}{Theorem}[section]
\newtheorem{proposition}[theorem]{Proposition}

\theoremstyle{definition}

\newtheorem{definition}[theorem]{Definition}

\newcommand{\R}{\mathbb{R}}

\newcommand{\M}{\mathcal{M}}
\def\P{{\mathcal P }}
\def\RF{\rho_{F}}     
\def\RW{\rho_{W}}  

\newcommand\blue[1]{\textcolor{blue}{#1}}

\title{Robust Policy Selection and Harvest Risk Quantification for Natural Resources Management under Model Uncertainty} 

\author[a,b]{Georgios I. Papayiannis} 

\affil[a]{\footnotesize Hellenic Naval Academy, Department of Naval Sciences, Section of Mathematics, Mathematical Modeling and Applications Laboratory, Piraeus, GR}
\affil[b]{\footnotesize Athens University of Economics \& Business, Department of Statistics, Stochastic Modeling and Applications Laboratory, Athens, GR}

\graphicspath{ {images/} } 

\begin{document}

\AddToShipoutPictureBG*{%
  \AtPageUpperLeft{%
    \hspace{\paperwidth}%
    \raisebox{-\baselineskip}{%
      \makebox[0pt][r]{\blue{Accepted in Journal of Dynamics and Games} }
}}}%

\maketitle

\begin{abstract}
In this work the problem of optimal harvesting policy selection for natural resources management under model uncertainty is investigated. Under the framework of the neoclassical growth model dynamics, the associated optimal control problem is investigated by introducing the concept of model uncertainty on the initial conditions of the operational procedure. At this stage, the notion of convex risk measures, and in particular the class of Fr\'echet risk measures, is employed in order to quantify the total operational and marginal risk, whereas simultaneously obtaining robust to model uncertainty harvesting strategies.    
\end{abstract}


\section{Introduction}

Optimal decision making is traditionally an extremely active field in economics, while recently, there is a growing interest in natural resources management \cite{boucekkine2013, brock2014a, brock2014b, brock2014c, clark2010, fabbri2020, xepa2016}. The focus has moved to the study of optimal control problems taking into account uncertainty issues \cite{baltas2018, baltas2019robust, baltas2022, bielecki2019, li2012} which occur very often in practice, either because the exact evolution law of the quantities of interest (e.g. the evolution of natural capital in terms of environmental economics) is not possible to be precisely modeled, or exogenous factors act to the underlying dynamics causing random disturbances. Therefore, the treatment of uncertainty issues related to the decision making process attracts more and more the interest of the scientific community.  

From the perspective of decision making, the classical framework treating the uncertainty issue puts the decision maker into the position to provide her/his aversion preferences with respect to a certain provided/estimated model either through the multiplier preferences principle or through the constraint preferences principle \cite{hansen2001robust, maccheroni2006}. Modern approaches in decision theory investigate the tricky situation caused by the more challenging situation referred to as {\it model uncertainty} or {\it model ambiguity}, i.e. the case where a multitude of models are available, which may be divergent, their validity is not known and of course the decision maker cannot trust only one of them but should merge the available information in order to make a robust decision, with respect to model uncertainty. This task is not straightforward since several issues have to be resolved. First of all, a very important and clearly non trivial matter is to choose the appropriate metric tools to measure distances between the available models. In general, since these models are typically provided in terms of probability measures, appropriate metrization instruments should respect and be compatible with the topology of the space of probability measures in order to provide valid quantifications and solutions. However, the space of probability measures does not carry the convenient structure of a vector space, therefore the task of providing appropriate metrics with good properties and general applicability remains an interesting and important problem. Secondly, the task of aggregating the available information from the various models with an efficient manner is also a challenging one. Several approaches have been discussed in the literature where the most recent and quite popular is that of employing the notion of Fr\'echet mean or Fr\'echet barycenter \cite{frechet1948} in order to determine the mean element even in spaces where the typical notion of euclidean mean is not applicable. A very interesting approach in optimal decision making relying on this concept is proposed in \cite{petracou2021decision} where a very flexible and extendable framework to treat model uncertainty issues is analyzed. Lastly, the issue of model validity is something that someone has to be aware of, especially in cases where the decision making process is performed under a dynamic setting and possibly during the process new evidence become available. In this direction, learning schemes that determine the optimal weight allocation to each model, taking into account criteria assessing the models' performance comparing to the recorded data, are proposed and investigated under different metric tools in \cite{papayiannis2018learning}.

Risk quantification in optimal policy selection problems is of great importance, especially in the setting of natural resources management. Ideally, a natural resources manager would like to design an optimal harvesting policy on a specific domain of activity, taking into account both the system's conservation/sustainability and of course the level of turnout that is succeeded. The framework of convex risk measures \cite{follmer2002convex, frittelli2002} provides a worst-case scenario risk perspective, with respect to the factors that affect randomly the optimal control problem and allows the decision maker to provide his/hers risk aversion preferences and appropriately quantifying the related risk output. Recently, Papayiannis and Yannacopoulos \cite{papayiannis2018convex} proposed an extension of the convex risk measures framework that is compatible with cases of model uncertainty as described above, providing risk calculations on the rationale of subjective worst-case scenario but robust to model ambiguity. 

In this paper, an optimal policy selection problem concerning natural resources management (e.g. water resources, oil resources, fisheries, etc) is discussed within the framework of model uncertainty. Model uncertainty concerns the incomplete knowledge or inability of observation concerning certain parameters values which affect the spatio-temporal dynamics of the natural capital of interest. Therefore, the interest is twofold: (a) deriving the optimal control of the problem, and simultaneously (b) treating robustly the model uncertainty situation. First, the general framework and approaches are discussed in Section \ref{sec-2} and then, as a particular case of interest, the optimal harvesting problem on a discretized version of the neoclassical growth model \cite{acemoglu2009} under model uncertainty on the initial states of the natural capital at each sub-domain of interest is studied. In Section \ref{sec-3}, the optimal control problem is studied separately and the optimal harvest rate strategy is derived in analytic form. In Section \ref{sec-4}, the risk estimation and risk allocation problem is considered, providing explicit formulas for the calculation of the total and individual harvest risks under model uncertainty. Moreover, robust optimal harvest strategies are obtained in closed form formulas, employing the discussed risk quantification framework in order to treat robustly the model ambiguity issue on a human against nature game-oriented setting.

\section{Statement of the Problem}\label{sec-2}

Consider the case that a natural resources manager needs to derive an optimal policy regarding the harvesting of a physical capital, e.g. water withdrawal, oil withdrawal, biomass, etc., spatially located on a physical domain $\Omega$ a certain time interval $[0,T]$ where $T>0$ denotes the operation's time horizon. Depending the nature of the operation, the harvesting activities should be discounted by a utility function $U$ during the operational time interval and a function $g$ on the time horizon $T$ measuring the condition of the natural capital system. These functions quantify the payoff at each time instant depending on the harvest rate, let us denote it by $c(t,x)$, and the condition of the system producing the physical capital at $T$, $k(T,x)$ for all $x\in\Omega$. The latter criterion, should introduce to the problem's objective, sustainability preferences or constraints that should be taken into account from the optimal policy to be chosen. Clearly, the above criteria depend on the spatio-temporal dynamical system describing the density of the physical quantity of interest, which in general terms could be of the form
\begin{eqnarray}\label{dyn-0}
\frac{\partial}{\partial t}k(t,x) = F(t, x, k(t,x), c(t,x)), \,\,\,\,\, k(0,x)=k_0,\,\,\,\,\, \forall x\in \Omega, \,\,\, t\in[0,T]
\end{eqnarray}
where $k(t,x)$ denotes the physical capital density at point $(t,x)$, being the {\it state} variable of the system, $c(t,x)$ denotes the harvest rate at point $(t,x)$ being the {\it control} variable of the system, and $F$ the function which determines the changes on the physical quantity density at each point $(t,x)$ depending both on the state and control of the system and on other parameters modeling known factors' contributions. In the case where a plausible model for $F$ is available (for simplicity assume for the time being that this model is of a deterministic type), then the manager has to solve the optimal control (policy selection) problem
\begin{eqnarray}\label{opt-0}
\begin{array}{l}
\max_{ c \in \mathcal{H} } \left\{ \int_0^T e^{-rt} U( c(t,x) ) dt + g(k(T,x), T) \right\}, \,\,\,\, \forall x\in\Omega \\
\mbox{subject to dynamics of equation \eqref{dyn-0}} 
\end{array}
\end{eqnarray}
where $r>0$ is the discount rate (assumed constant) and $\mathcal{H}$ denotes a suitable Hilbert space (e.g., $L^2$) for the optimal controls. The above problem is some cases admits explicit solutions depending on the forms of the utility and payoff functions $U,g$. However, the assumption that the  system \eqref{dyn-0} is perfectly modeled lacks realism in general. Under most circumstances, uncertainty is presented to various aspects of the dynamical system, e.g., on the initial conditions which are usually not precisely known, to certain parameters of the drift term $F$ presenting fluctuations due to the occurrence of various random phenomena (climate conditions variations, occurrence of random spatial events, etc.). It is evident that a natural resources manager should take into account these uncertainty issues in order to derive efficient, realistic and robust policies which will not threaten possible future operations and to obtain the required sustainability levels. In this perspective, the optimal control problem needs to be reshaped into a form that takes into account and treats the discussed uncertainty issues. 

Assume that the random behaviour of the parameters in \eqref{dyn-0}, which are subject to uncertainty, is captured by some probability model $Q_0$ provided by some expert. The manager may not desire to employ directly the model $Q_0$ to derive her/his strategy but he/she may deviate from it, depending on the level of reliability that is allocated to the agent or the source that provides the information. Attempting to build such a decision framework, the concept of {\it convex risk measures} \cite{follmer2002convex} is employed. Any convex risk measure is obtained as a worst-case scenario output of a relevant maximization problem given a suitable loss function depending on certain risk factors/stressors, the random behaviour of which is described by the provided probability models. In the case under discussion, a loss function might be a function of the optimal harvest rate $c^*(t,x)$ related to problem \eqref{opt-0}. However, due to the risk factors effects, let us denote them by $Z = (Z_1, Z_2, ..., Z_d)'$ (considering them as $d$ random variables), the harvest rate, obtained as the optimal solution of problem \eqref{opt-0}, depends on the realizations of the random variable $Z$ and in order to clarify this dependence we may rewrite $c^*(t,x)$ as $c^*(t,x;Z)$. Then, an appropriate loss function could be of the form
\begin{equation}
L(Z) := \frac{1}{T}\int_0^T \frac{1}{|\Omega|}\int_{\Omega} \ell( c^*(t, x;Z)) dx \, dt
\end{equation} 
with $\ell:\R\to\R$ being an integrable function calculating the loss output at each point $(t,x)$, and $Z \sim Q$ the random variables affecting the dynamics of \eqref{dyn-0} with $Q$ being the provided/selected probability model describing the random behaviour of the risk factors $Z$. Under the perspective of convex risk measures, a robust estimate of the risk $L(Z)$ is provided by the worst-case scenario loss, expressed as the maximum expected loss over a set $\mathfrak{Q}$ of plausible probability measures that may act as possible models for the risk factors $Z$. In particular, the harvest risk related to the control problem \eqref{opt-0} is determined through the solution of the maximization problem
\begin{eqnarray}\label{risk-1}
\rho(L) = \mathbb{E}_{Q^*}[ L(Z) ] :=  \max_{Q \in \mathfrak{Q}} \mathbb{E}_{Q}[ L(Z) ]
\end{eqnarray}
where the set $\mathfrak{Q}$ is constructed according to the model uncertainty preferences of the decision maker with respect to a provided model $Q_0$, while the optimizer $Q^*$ can be employed  to provide the related robust optimal harvest strategy of \eqref{opt-0} defined as 
\begin{equation}\label{rob-0}
\widetilde{c}(t,x) := c^*(t,x;Z^*), \,\,\,\, Z^* \sim Q^*.
\end{equation}  
On problem \ref{risk-1}, the manager's preferences, which are taken into account on the derivation of $Q^*$, are translated to deviations from the provided probability model $Q_0$, so if no deviations at all are desired then it is clear that $Q^* = Q_0$. Alternatively, the model space $\mathfrak{Q}$ is formulated as
\begin{equation}
\mathfrak{Q} := \{ Q\in\mathcal{P}(\R^d)\,\,\, : \,\,\, d(Q, Q_0) \leq \eta \}
\end{equation}
where $d$ is an appropriate metric in the space of probability measures on $\R^d$, $\mathcal{P}(\R^d)$ e.g., Wasserstein distance \cite{papayiannis2018convex, petracou2021decision, santambrogio2015optimal, villani2021topics}, and $\eta>0$ is the maximum desired deviation level from model $Q_0$ in terms of the metric sense that is used. The optimization problem in \eqref{risk-1} can be equivalently expressed in a {\it multiplier-preferences} formulation instead of the current {\it constraint-preferences} representation by
\begin{eqnarray}\label{risk-2}
\rho(L) = \mathbb{E}_{Q^*}[ L(Z) ] :=  \max_{Q \in \mathcal{P}(\R^d)} \left\{ \mathbb{E}_{Q}[ L(Z) ] - \frac{1}{\gamma}d(Q,Q_0) \right\}
\end{eqnarray}
for an appropriate choice of the aversion parameter $\gamma>0$ (quantifying the manager's aversion preferences from $Q_0$). This concept is further generalized and extended to the very interesting case where a multitude of models is provided for describing the behaviour of random variable $Z$ possibly divergent and with varying probabilities of realization \cite{papayiannis2018convex, petracou2021decision}. In optimal harvesting problems, different probability models could identify different scenarios about the conditions in the spatio-temporal domain that the natural resources optimal management problem is considered. 

Following the discussion above, a robust treatment of a natural resources management problem is condensed to three steps: (i) identification of the risk factors affecting the system evolution and to obtain some model/-s describing these random behaviours, (ii) derivation of the optimal aggregate model describing the random part of the problem taking into account the manager's preferences and (iii) to derive (if it is possible in explicit form) the optimal harvest rates at each spatio-temporal location with a robust manner, i.e. discounted by the probability measure $Q^*$. In a more compact form, the problem under study may be stated as
\begin{eqnarray}\label{opt-2}
\begin{array}{l}
\max_{Q \in \mathfrak{Q}} \mathbb{E}_{Q}[ L(Z) ] \\
\mbox{subject to:} \\
c^* \mbox{ maximizer of \eqref{opt-0} for any } Z \in \R^d, \,\,\, Z\sim Q 
\end{array}
\end{eqnarray}
Note than in the above formulation are included both the one prior model case and the multitude of models case since the model space $\mathfrak{Q}$ is flexibly determined depending on the provided information and the manager's preferences. 

In next sections, a version of problem \eqref{opt-2} is examined where the derivation of solutions in explicit form both for the optimal probability measure related to the risk estimation problem \eqref{opt-2} and the optimal harvest rate related to the inner optimal control problem \eqref{opt-0} is possible, enabling through combination the derivation of robust optimal strategies in the sense of \eqref{rob-0}.

\section{A Natural Resources Optimal Harvesting Problem and its Associated Optimal Controls}\label{sec-3}

In this section a specific case of an optimal control problem for natural resources management is examined. In particular, the Neoclassical growth model is adopted as the working framework by dividing the dynamics to regions of interest by appropriate spatial averaging. Then, solutions for the optimal control problem are obtained in closed form.

\subsection{A discretized formulation of the optimal control problem}

Assume that the spatio-temporal density $k(t,x)$ of the physical capital of interest is described by the dynamics
\begin{eqnarray}\label{dyn-00} 
\left\{
\begin{array}{l}
\frac{\partial}{\partial t} k(t,x) = (\mathcal{L} + A)\, k(t,x) - B\, c(t,x), \\
k(0,x)=k_0, \,\,\, x\in\Omega, \,\,\, t \in [0,T]
\end{array}
\right.
\end{eqnarray}
where $\mathcal{L}$ denotes the Laplacian operator, $A, B$ are some linear operators, $c(t,x)$ denotes the harvest rate (or consumption rate) at point $(t,x)$ and $k_0$ denotes the initial state of the physical capital at each spatial location. Spatial dependence of the problem can be allocated to certain sub-regions, let us say $N$ in number, dividing the initial domain of interest $\Omega$ to $N$ sub-domains $\Omega_i$ such that $\Omega_i \cap \Omega_j = \O$ for $i \ne j$ and $\cup_{i=1}^N \Omega_i = \Omega$. In each sub-region $i$, the physical quantity's density at time $t$ is represented by averaging over the spatial domain $\Omega_i$ the natural capital concentration $k(t,x)$, obtaining the expressions
\begin{eqnarray}
k_i(t) := \frac{1}{|\Omega_i|}\int_{\Omega_i} k(t,x) dx, \,\,\,\, i=1,2,...,N
\end{eqnarray} 
where similarly, the local harvest rates are represented in averaged formulations by
\begin{eqnarray}
c_i(t) := \frac{1}{|\Omega_i|}\int_{\Omega_i} c(t,x) dx, \,\,\,\, i=1,2,...,N.
\end{eqnarray} 
In this formulation, for each time instant $t$ the physical quantity at each sub-region is described in vector form
\begin{equation}
k(t) = \begin{pmatrix} k_1(t), & k_2(t), & \ldots, & k_N(t) \end{pmatrix}',  
\end{equation}
the harvest intensity by
\begin{equation}
c(t) = \begin{pmatrix} c_1(t), & c_2(t), & \ldots, & c_N(t) \end{pmatrix}'. 
\end{equation}
Upon the aforementioned discretization through spatial averaging, the resulting model describing the natural capital dynamics is expressed as
\begin{eqnarray}\label{dyn-1}
\left\{ \begin{array}{l}
\frac{\partial}{\partial t} k_i(t) = \left( ( \mathcal{L}_G+A_{D}) \, k(t) \right)_i - B_{D,i}\, c_i(t), \\
k_i(0) = k_{0,i}\in \mathbb{R} \end{array} \right. \,\,\,\, i=1,2,...,N,
\end{eqnarray}
where $\mathcal{L}_G$ denotes the graph Laplacian matrix of dimension $N \times N$ being the discrete analog to the Laplacian operator $\mathcal{L}$. Identifying the divided spatial domain $\Omega$ with an undirected graph $\mathcal{G} = (\mathcal{V},\mathcal{E})$ with the set of vertices $\mathcal{V} = \{1,2,...,N\}$ representing the sub-domains $\{\Omega_1, ..., \Omega_N\}$ and the set of edges $\mathcal{E}$ representing the communications/interactions between the sub-domains, then using symmetric real valued weights $w_{ij}=w_{ji}\in\R$ assessing the intensity of interaction between sub-domains $\Omega_i$ and $\Omega_j$, the graph Laplacian matrix is defined as the matrix $\mathcal{L}_G$ with elements
\begin{equation}
\mathcal{L}_{G,ij} = \left\{ 
\begin{array}{ll}
- w_{ij}, & \mbox{if } i \mbox{ communicates with } j\\
\sum_{k \sim i} w_{ik}, & i=j\\
0, & \mbox{otherwise} 
\end{array} \right.
\end{equation} 
where the summation is performed over the edges (sub-domains) that communicate. In the simplest case, the graph Laplacian weights may be chosen as $w_{ij}=1$ when two vertices are connected and $w_{ij}=0$ otherwise, while the graph Laplacian matrix can be recovered by the decomposition $\mathcal{L}_G = D-A$ with $A$ being the symmetric {\it adjacency matrix} containing the weights $w_{ij}$, and $D$ being the diagonal {\it degree matrix} where the $i-th$ diagonal element is the total number of vertices that communicate with vertex $i$. As a result, the graph Laplacian matrix $\mathcal{L}_{G}$, describes the natural capital flows and interactions between and within the sub-domains $\Omega_i$ for $i=1,2,...,N$ by a graph-type model. Moreover, the rest matrices in \eqref{dyn-1}
\begin{eqnarray}
A_{D} = diag( A_1, A_2, ..., A_N)', \,\,\,\, B_{D} = diag( B_1, B_2, ..., B_N)'
\end{eqnarray} 
are the discrete versions, obtained through spatial averaging, of the linear operators $A, B$ introduced in \eqref{dyn-00}. Clearly, in each sub-domain $\Omega_i$, the $i-th$ diagonal elements of the matrices $A_D, B_D$ correspond to the averaged spatial effects of the current sub-region related to the averaged natural capital density $k_i$ and the averaged harvest rate $c_i$, respectively, while $k_0 \in \R^N$ denote the initial density on each sub-domain of $\Omega$.

Under the framework of the Neoclassical growth model \cite{acemoglu2009, solow1999ngm} the optimal harvest rate can be chosen through the solution of the maximization problem
\begin{eqnarray}\label{opt-3}
\begin{array}{l}
\max_{ c_1(\cdot), c_2(\cdot), ..., c_N(\cdot) \in \mathcal{H}_N } \left\{ \int_0^T e^{-rt} \frac{ \left( \sum_{i=1}^N D_i c_i(t)\right)^{1-\beta}}{1-\beta} dt + g(T, k(T)) \right\}\\
\mbox{subject to dynamics of equation \eqref{dyn-1}}  
\end{array}
\end{eqnarray}
where $U(c)$ is replaced by a constant relative risk aversion (CRRA) type utility function with $\beta \in (0,1)$ denoting the elasticity parameter, $D_i$ for $i=1,2,...,N$ are spatial-dependent parameters (e.g., weights allocated by the manager to each sub-region), $g$ is the associated payoff function concerning the natural capital density at the terminal horizon $T$ and $\mathcal{H}_N$ is an appropriate space containing possible solutions (controls/harvest rates) of the problem depending only on time. For solvability reasons, let us further assume that $g(k) = \frac{e^{-rT}}{1-\beta} (\sum_{i=1}^N E_i k_i(T))^{1-\beta}$ with $E_i, \,\, i=1,2,...,N$ denoting the importance allocated by the manager to the natural capital state at time $T$ at each domain $\Omega_i$. In order to simplify notation, from now on we use the notation $\langle D, c \rangle := \sum_{i=1}^N D_i c_i(t)$ and $\langle E, k \rangle := \sum_{i=1}^N E_i k_i(t)$ introducing discretized versions of the inner products $\int_{\Omega} D(x) c(t,x) dx$ and $\int_{\Omega} E(x) k(t,x) dx$ performing averaging on the spatial effects of the operators $D, E:\mathcal{H} \to \R$ per sub-region (which is much preferable and applicable in practice). 

\subsection{Derivation of the optimal controls}

For solvability in analytic form let us assume that $E = \langle D, k_0\rangle^{\beta/(1-\beta)}$. Following the problem formulation in \eqref{opt-3}, the associated Hamilton-Jacobi-Belmann equation is
\begin{eqnarray}\label{hjb}
\begin{array}{l}
-\frac{\partial}{\partial t}v = \sup_{c \in \mathcal{\widetilde{H}}}\left\{ \langle (\mathcal{L}_G + A_D)\, k(t), \mathcal{D}v \rangle - \langle \sum_{i=1}^N B_{D,i} c_i(t), \mathcal{D}v\rangle + e^{-rt} \frac{\langle D, c(t) \rangle^{1-\beta}}{1-\beta} \right\}, \\
v(T,k(T)) = g(T,k(T)) = e^{-rT} \frac{\langle E,  k(T) \rangle^{1-\beta}}{1-\beta}
\end{array}
\end{eqnarray}
where $v(t,k)$ is the value function of the problem, $\mathcal{D}v$ the derivative of the value function with respect to $k$, i.e. $\mathcal{D}v = ( \frac{\partial}{\partial k_1}v,  \frac{\partial}{\partial k_2}v, \ldots, \frac{\partial}{\partial k_N}v )'$, $\langle \cdot, \cdot \rangle$ denotes the inner product operation and $\mathcal{\widetilde{H}}$ is an appropriate space for the controls. 

The following proposition, state the solutions in explicit form for the value function satisfying the Hamilton-Jacobi-Belmann equation stated in \eqref{hjb} and the optimal harvest rates of the associated problem \eqref{opt-3}.  

\begin{proposition}\label{prop-1}
The value function satisfying the equation \eqref{hjb} in explicit form is 
\begin{equation}\label{ovf}
v(t,k(t)) = \frac{ e^{-rt} }{ 1-\beta } \left[ \left(k_0 - \frac{\Lambda(\alpha)}{\theta}\right) e^{-\theta(T-t)} + \frac{\Lambda(\alpha)}{\theta} \right]^{\beta} \langle \alpha, k(t) \rangle^{1-\beta}
\end{equation}
while the optimal harvest rates of problem \eqref{opt-3} are 
\begin{equation}\label{ohr}
c^*_i(t) = \alpha_i \left( \frac{B_{D,i}}{D_i} \right)^{-1/\beta}\left[ \left(k_0 - \frac{\Lambda(\alpha)}{\theta}\right) e^{-\theta(T-t)} + \frac{\Lambda(\alpha)}{\theta} \right]^{-1} \langle \alpha,  k(t) \rangle
\end{equation}
for $ i=1,2,...,N$, where $\alpha = (\alpha_1, \alpha_2, ...., \alpha_N)'$ being the eigenvector corresponding to the lowest eigenvalue of the eigenvalue problem $(\mathcal{L}_G + A_D)\alpha = \lambda \alpha$ with $\|\alpha\|=1$, $\theta = (r-\lambda(1-\beta))/\beta$ and $\Lambda(\alpha) = \sum_{i=1}^N B_{D,i} (B_{D,i}/D_i)^{-1/\beta} \alpha_i^{(\beta-1)/\beta}$. 
\end{proposition}

\begin{proof}[Proof of Proposition \ref{prop-1}]

Working with problem's \ref{opt-3} associated HJB equation \eqref{hjb} in order to obtain the optimal harvest rates one just needs to solve the problem 
$$\sup_{c_1(\cdot),...,c_N(\cdot)} \left\{ \frac{e^{-rt}}{1-\beta}\langle D, c(t) \rangle^{1-\beta} - \langle \sum_{i=1}^N B_{D,i} \, c_i(t), \mathcal{D}v \rangle \right\}$$ 
since the remaining terms do not depend on the harvest rates. Since the equation is concave with respect to $c$, it suffices to solve the first order conditions to obtain a characterization for the maximizers. Indeed, taking first order conditions with respect to each $c_i$, corresponding to the spatial average on each sub-domain $\Omega_i$, we conclude to the optimal harvest rates of problem \eqref{opt-3} with respect to the value function $v(t,k(t))$, which are 
\begin{equation}\label{ohr-1}
c_i^*(t) = e^{-rt/\beta} \left( (B_{D,i}/D_i) \,\, \mathcal{D} v_i(t) \right)^{-1/\beta}  
\end{equation}
for $i=1,2,...,N$. Assume that a suitable form for a candidate value function satisfying HJB equation \eqref{hjb} is the following
\begin{equation}\label{eq-0}
\hat{v}(t,k) =  \frac{ \psi_0(t)^{1-\beta} }{ 1-\beta }\langle \alpha, k(t) \rangle^{1-\beta}
\end{equation}
with $\alpha = (\alpha_1, \alpha_2, ...., \alpha_N)'$ being the eigenvector corresponding to the lowest eigenvalue of the eigenvalue problem $(\mathcal{L}_G+A_D)\alpha = \lambda \alpha$ with $\|\alpha\|=1$ and $\psi_0$ being a sufficiently smooth function to be determined in the sequel. Using this {\it guess}, each one of the associated HJB equation terms are calculated:
\begin{eqnarray}
&& \frac{\partial}{\partial t}v = \psi_0(t)^{-\beta} \psi_0'(t) \langle \alpha,  k(t) \rangle^{1-\beta}\\
&& (\mathcal{D}v)_i := \frac{\partial}{\partial k_i}v = \psi_0(t)^{1-\beta} \langle \alpha,  k(t) \rangle^{-\beta} \alpha_i, \,\,\,\, i=1,2,...,N
\end{eqnarray}
where the latter can be expressed in a more compact form by
\begin{equation}\label{DV}
\mathcal{D}v = \psi_0(t)^{1-\beta} \langle \alpha,  k(t) \rangle^{-\beta} \alpha.
\end{equation}
Calculating the diffusion term of \eqref{hjb} we get:
\begin{eqnarray}\label{hjb-3}
\langle (\mathcal{L}_G + A_D)\, k(t), \mathcal{D}v \rangle &=& \langle k(t), (\mathcal{L}_G + A_D)^* \mathcal{D}v \rangle \nonumber \\
					&=& \psi_0(t)^{1-\beta} \langle \alpha, k(t) \rangle^{-\beta} \langle k(t), (\mathcal{L}_G + A_D)^* \alpha \rangle \nonumber \\
					&=& \lambda \psi_0(t)^{1-\beta} \langle \alpha, k(t) \rangle^{-\beta} \langle k(t), \alpha \rangle  \nonumber \\
					&=& \lambda \psi_0(t)^{1-\beta} \langle k(t), \alpha \rangle^{1-\beta} 
\end{eqnarray}
where $\lambda$ are the eigenvalues of $(\mathcal{L}_G + A_D)$. Using the obtained form of the maximizers in \eqref{ohr-1} and substituting each HJB calculated term in \eqref{hjb} one obtains
\begin{equation}\label{hjb-1}
-\psi_0(t) = \Lambda(\alpha) \psi_0(t) + e^{-rt/\beta} \frac{\beta}{1-\beta} \Lambda(\alpha) \psi_0(t)
\end{equation} 
where 
\begin{equation}\label{Lama}
\Lambda(\alpha) = \sum_{i=1}^N B_{D,i} (B_{D,i}/D_i)^{-1/\beta} \alpha_i^{(\beta-1)/\beta}.
\end{equation}
Performing the transformation $\psi_0(t) = e^{\mu t}\varphi_0(t)$ we have $\psi'_0(t) = \mu e^{\mu t} \varphi_0(t) + e^{\mu t} \varphi'_0(t)$ and substituting in \eqref{hjb-1} we get
\begin{equation}
-\mu e^{\mu t} \varphi_0(t) - e^{\mu t} \varphi'_0(t) = \lambda e^{\mu t} \varphi_0(t) + \frac{\beta}{1-\beta} \Lambda(\alpha) e^{ -\frac{rt}{\beta} + \frac{\mu (2\beta-1)t}{\beta} } \varphi_0(t)^{\frac{(2\beta - 1)}{\beta}}
\end{equation}
Choosing $\mu$ such that $\mu = -r/(1-\beta)$ we retrieve the folowing Bernoulli-type equation
\begin{equation}
-\varphi_0'(t) = \left( \lambda - \frac{r}{1-\beta} \right) \varphi_0(t) + \frac{\beta}{1-\beta} \Lambda(\alpha) \varphi_0(t)^{(2\beta-1)/\beta} 
\end{equation}
Since a Bernoulli equation of the form $x'(t) = c_1 x(t) + c_2 x(t)^{\gamma}$ admits the explicit solution $x(t)=\left\{C e^{c_1 (1-\gamma) t} - c_2/c_2 \right\}^{1/(1-\gamma)}$, choosing $\gamma = (2\beta-1)/\beta$ we obtain the solution
\begin{equation}
\varphi_0(t) = \left\{ C e^{\frac{r-\lambda(1-\beta)}{\beta}t} + \frac{ \beta \Lambda(\alpha) }{r-\lambda(1-\beta)} \right\}^{\beta/(1-\beta)}
\end{equation}
and performing the inverse transform we get $\psi_0(t)$ as
\begin{equation}\label{eq-1}
\psi_0(t) = e^{-rt/(1-\beta)} \left\{ C e^{\frac{r-\lambda(1-\beta)}{\beta}t} + \frac{ \beta \Lambda(\alpha) }{r-\lambda(1-\beta)} \right\}^{\beta/(1-\beta)}
\end{equation}
The terminal condition $\psi_0(T)^{1-\beta} = e^{-rT}k_0^{\beta}$ is employed to calculate $C$ which results to
\begin{equation}\label{eq-2}
C = \left( k_0 - \frac{\beta \Lambda(\alpha)}{r-\lambda(1-\beta)} \right) e^{-\frac{r-\lambda(1-\beta)}{\beta}T}
\end{equation}
Then substituting \eqref{eq-2} to \eqref{eq-1} we get that
\begin{equation}\label{eq-3}
\psi_0(t) = e^{-rt/(1-\beta)}\left[ (k_0-\Lambda(\alpha)/\theta)e^{-\theta (T-t)} + \Lambda(\alpha)/\theta \right]^{\beta/(1-\beta)}
\end{equation}
where $\theta := (r-\lambda(1-\beta))/\beta$. Then, substituting \eqref{eq-3} to \eqref{eq-0} we obtain the stated result \eqref{ovf} in Proposition \ref{prop-1}. Also, combining equations \eqref{DV} and \eqref{eq-3} and substituting to \eqref{ohr-1} are obtained the optimal controls stated in \eqref{ohr} in Proposition \ref{prop-1}.
\end{proof}

Note that under the assumption that $\theta>0$ and for a terminal horizon $T$ large enough, the optimal harvest rates stated in Proposition \ref{prop-1} are asymptotically equal to 
\begin{equation}
c^*_i(t) \simeq \frac{ \alpha_i\theta}{\Lambda( \alpha)} (B_{D,i}/D_i)^{-1/\beta} \langle \alpha, k(t) \rangle 
\end{equation}
depending mainly on spatial characteristics (through the operators ($B_{D}, D$)) and the current states of natural capital at each location $k_i(t)$ for $i=1,2,...,N$. Then, working with the physical capital density dynamics using the optimal harvest rates we get:
\begin{eqnarray}
\frac{\partial}{\partial t}k_i(t) &=& [(\mathcal{L}_G + A_D)k(t)]_i - B_{D,i} c^*_i(t)\\
			&=& [(\mathcal{L}_G + A_D)k(t)]_i - \frac{\theta}{\Lambda(\alpha)} \Phi(t) B_{D,i} (B_{D,i}/D_i)^{-1/\beta}\langle \alpha, k(t) \rangle \alpha_i 
\end{eqnarray}
where $\Phi(t) := \left[ 1 + \left(\frac{\theta k_0}{\Lambda(\alpha)}-1\right)e^{-\theta(T-t)} \right]^{-1}$. So, the physical capital dynamics can be written in matrix form as
\begin{equation}
\frac{\partial}{\partial t} k(t) = (\mathcal{L}_G  + A_D) k(t) - \frac{\theta}{\Lambda(\alpha)}\Phi(t) B_D \widetilde{B} A \widetilde{A} k(t)
\end{equation}
where
\begin{eqnarray}
&& \widetilde{B} := diag( (B_{D,1}/D_1)^{-1/\beta}, (B_{D,2}/D_2)^{-1/\beta}, ..., (B_{D,N}/D_N)^{-1/\beta})\\
&& A := diag(\alpha_1, \alpha_2, ..., \alpha_N)
\end{eqnarray}

\begin{equation}
\widetilde{A} := \begin{pmatrix} \alpha_1 & \alpha_2 & \ldots  & \alpha_N\\
										   \alpha_1 & \alpha_2 & \ldots  & \alpha_N\\
										   \vdots     & \vdots     & \ddots & \vdots \\
										    \alpha_1 & \alpha_2 & \ldots  & \alpha_N \\
						\end{pmatrix} 
\end{equation}
Notice again that taking $T\to\infty$ it holds that $\Phi(t) \to 1$ simplifying the physical capital dynamics to the form
\begin{equation}
\frac{\partial}{\partial t} k(t) = (\mathcal{L}_G  + A_D) k(t) - \frac{\theta}{\Lambda(\alpha)}B_D \widetilde{B} A \widetilde{A} k(t)
\end{equation}
In this case, setting 
\begin{equation}\label{Matrix}
M := (\mathcal{L}_G  + A_D) - \frac{\theta}{\Lambda(\alpha)}B_D \widetilde{B} A \widetilde{A},
\end{equation} 
the physical capital density $k(t)$ can be described by the matrix exponential  
\begin{equation}
k(t) = e^{t M} k_0
\end{equation}
where $k_0 = \sum_{i=1}^N k_{0,i} \, e_i$ and similarly, the optimal harvest rates by the equation
\begin{equation}\label{ohr-2}
 c(t) = \frac{\theta}{\Lambda(\alpha)}\widetilde{B} A \widetilde{A} k(t) = \frac{\theta}{\Lambda(\alpha)}\widetilde{B} A \widetilde{A} e^{t M} k_0
\end{equation}

\section{Harvest Risk Estimations and Robust Harvesting Policies under Model Uncertainty}\label{sec-4}

Consider again problem \eqref{opt-2} with the role of the random variable $Z$ to be played by the initial states of the natural capital $k_0$ which is now considered to be uncertain. In practice, this a very realistic assumption since in general $k_0$ is not known precisely and maybe subject to stochastic fluctuations, leading to stochastic fluctuations of the turnout related to the harvest rate. Ambiguity follows from the inability of exact modeling or measurement of the various exogenous factors (e.g. environmental conditions) affecting the distribution of $k_0$. The decision maker would like to evaluate possible losses from such fluctuations taking into account the model uncertainty and the discussed class of risk measures offers such an evaluation. To provide a specific example, consider that the total loss from harvesting operations of a firm on the domain of interest $\Omega$ is provided by the function
\begin{equation}
L := \frac{1}{T}\sum_{i=1}^N \pi_i L_i = -\frac{1}{T}\sum_{i=1}^N \pi_i \int_0^T c_i(t; k_0)dt
\end{equation}
with $\pi \in [0,1]$ and $\sum_i \pi_i=1$ denoting the importance allocated by the firm to its harvesting activities on the sub-domains $\Omega_1,...,\Omega_N$ (since the typical risk measures framework requires loss maximization, $L$ is defined as such a loss from the harvesting operations in the domain, which is equivalent to the profit minimization from the related harvesting operation). Assume that the firm needs to determine the harvest risk for a specified time horizon $T>0$ (large enough) and under the occurring uncertainty regarding the initial states of the physical capital at each region $\Omega_i$. Since several scenarios can occur depending on the environmental conditions and other exogenous factors affecting the initial states, there in not only one probability model that can plausibly describe the states $k_0 = (k_{0,1}, k_{0,2}, ..., k_{0,N})'$, but rather a multitude of different and possibly diverging models, let us say $n$ in number, with different probability of occurrence (or degree of realism) which is allocated/determined by the risk manager of the firm as the weighting vector $w\in\Delta^{N-1}$. In this case, the risk mapping of the firm related to the harvesting activities (quantifying potential losses) is expressed by the function:
\begin{equation}\label{risk-map}
L := \Phi_0(k_0) = \sum_{i=1}^N \pi_i \Phi_{0,i}(k_0) = \sum_{i=1}^N \pi_i G_i \langle \widetilde{\alpha}, k_0 \rangle
\end{equation}
where $G_i = \frac{\alpha_i \theta}{T \Lambda(\alpha)} (B_{D,i}/D_i)^{-1/\beta}$, $\widetilde{\alpha} = -\left((e^{TM}-I)M^{-1}\right)^{-1} \alpha$ with $I$ being the identity operator and $M$ as defined in \eqref{Matrix}.

Since the modeling scheme for the natural capital dynamics after the discretization can be identified by the evolution of the capital at certain nodes of a graph, very appropriate probability models to capture the uncertainty on the initial states (i.e. initial condition on each node of the graph) may be members of the Location-Scatter family of probability models. Using such type of models for $k_0$, the ambiguity is transferred to some location parameters $m \in \R^N$ and to some dispersion parameters $S \in \mathbb{P}(N)$ \footnote{Where $\mathbb{P}(N)$ denotes the set of positive definite $N\times N$ matrices.} describing possible spatial dependencies between the graph nodes (i.e., transport phenomena between the various sub-regions of the domain $\Omega$) which fully characterize the probability model. Then, the provided prior set of models $\M=\{ Q_1,...,Q_N\}$ contains possible models for the random behaviour of the initial states $k_0$ belonging to the Location Scatter family, i.e. $Q_i = LS(m_i, S_i)$ for all $i=1,2,...,N$. 

According to the framework of {\it convex risk measures} \cite{follmer2002convex, frittelli2002} the risk related to the loss $L = \Phi_0(k_0)$, with $k_0 \sim Q$, can be calculated through the solution of the problem
\begin{eqnarray}\label{robust}
\rho(L) = \sup_{Q \in \mathcal{P}(\mathcal{S})} \left\{ \mathbb{E}_{Q}[L] - \alpha(Q) \right\}
\end{eqnarray}
where $\mathcal{P}(\mathcal{S})$ denotes the set of probability measures on the set $\mathcal{S}$ which contains the possible states of the world, and $\alpha: \mathcal{P}(\mathcal{S}) \to \R\cup\{ \infty\}$, which practically determines the risk measure, is a function that penalizes certain extreme scenarios described by any provided model $Q_0$ according to the manager's preferences. Clearly, on the discussed context, this framework is not applicable since the provided information consists of a collection of models $\M$ and not just one model. This multi-prior setting is robustly handled by the framework of Fr\'echet risk measures \cite{papayiannis2018convex}, which are a natural extension of the one prior case to the multi-prior setting, where the notion of the Fr\'echet mean \cite{frechet1948} is employed to determine an aggregate prior model using the concept of barycenter to condense and robustly represent the multiple information by a single model. The distance of a probability measure $Q$ from the prior set $\M$ is measured in terms of the Fr\'echet function $\sum_{i=1}^N w_i d^2(Q,Q_i)$, and the element of minimal distance from $\M$
\begin{equation}
Q_B := arg\min_{Q\in\P(\mathcal{S})} \sum_{i=1}^N w_i d^2(Q,Q_i)
\end{equation}
is referred to as the {\it Fr\'echet mean} (or barycenter) of $\M$, $d$ is an appropriate metric and $w\in\Delta^{N-1}$ denotes a weight vector\footnote{Where $\Delta^{N-1} := \{ w\in\R^N \,\,\, :\,\,\, \sum_{i=1}^N w_i = 1,\,\,\, w_i\geq 0, \,\,\, \forall i \}$.} representing the trustworthiness that the manager allocates to each information source that provided a model estimate. Dispersion in the space of probability measures $\P(\mathcal{S})$ is quantified through the {\it Fr\'echet variance} which is defined as the minimal value of the Fr\'echet function ${\mathbb F}_{\M}(Q) := \sum_{i=1}^N w_i d^2(Q,Q_i)$ obtained at $Q_B$, which is also used to formulate an appropriate penalty function $\alpha(\cdot)$ according to the robust representation of convex risk measures (equation \eqref{robust}), in order to penalize deviance from the prior set $\M$. This leads to the definition of the general class of Fr\'echet risk measures, while choosing metric $d$ to be equal to one of the Wasserstein distances \cite{santambrogio2015optimal, villani2021topics} (e.g. the $2$-Wasserstein distance) which is a true metric in the space of probability measures, the notion of the Wasserstein barycentric risk measures is recovered. For the convenience of the reader, the related definitions follow. 

\begin{definition}[Fr\'echet and Wasserstein barycentric risk measure \cite{papayiannis2018convex}]\label{frisk} 
Let $\alpha : \R \to \R_{+}\cup\{\infty\}$ an increasing function with $\alpha(0)=0$ and $\Phi_0  : \R^{d} \to \R_{+}$ a risk mapping w.r.t. the stochastic factors $Z:\mathcal{S}\to\R^d$. Given a prior set $\M$ for $Z$, a sufficiently smooth risk mapping $L=\Phi_0(Z)$ and a Fr\'echet function ${\mathbb F}_{\M}$ the related {\it Fr\'echet risk measure} is defined for any $\gamma\in(0,\infty)$ as
\begin{eqnarray}
\RF(L) := \sup_{Q \in \mathcal{P}(\mathcal{S})} \left\{ \mathbb{E}_{Q}[L] -\frac{1}{2 \gamma} \alpha ( {\mathbb F}_{\M}(Q) ) \right\}.
\end{eqnarray}
For a set of weights $w \in \Delta^{n-1}$, the {\it Wasserstein barycentric risk measure} $\RW$ is defined as  
\begin{eqnarray}\label{rm-1}
\RW(L) := \sup_{Q \in \P(\mathcal{S})} \left\{  \mathbb{E}_Q [L] - \frac{1}{2 \gamma} F_{\M}(Q) \right\},
\end{eqnarray}
where $F_{\M}(Q) := \sum_{i=1}^N w_i W_2^2(Q,Q_i) - V_{\M}$, $V_{\M}:= \inf_{Q \in \P(\mathcal{S})} \sum_{i=1}^N w_i W_2^2(Q,Q_i)$ and $W_2(\cdot, \cdot)$ denotes the quadratic Wasserstein distance between any two probability models in $\P(\mathcal{S})$.
\end{definition}

Note that any Fr\'echet risk measure is also a convex risk measure and the typical setting of convex risk measures is retrieved if the model set $\M$ is a singleton. The special case of the Wasserstein barycentric risk measure provided a very appropriate vehicle to condense the information provided by the prior set $\M$ to a single aggregate model through the notion of Wasserstein barycenter, working directly in the natural space of the probability measures. The maximizer of the relevant risk calculation problem \eqref{rm-1}, is a probability measure $Q_*$ depending on the prior set $\M$ but at the same time carries the manager's preferences through the sensitivity parameter $\gamma>0$ and the weights $w\in\Delta^{N-1}$. The choice of the sensitivity parameter $\gamma > 0$, quantifies the general level of trust that the manager allocates to the provided models in $\M$ while each weight in $w$ places separately to each model the manager's belief regarding the credibility or success of the model in describing the risk factors. As a result, the measure $Q_*$ could act as a very suitable worst-case type discounting mechanism in terms of model uncertainty for future losses or payoffs from the operational activities of the firm (harvesting). Conveniently, under the family of Location-Scatter type of probability measures, the derivation of closed-form or semi-analytic form of solutions is possible in many cases. Concerning the harvesting problem under study, explicit formulas for both harvest risk calculations and aggregate probability model characterizations are derived. Moreover, the risk can be allocated to each sub-domain $\Omega_i$ indicating the contribution of each region to the total harvest risk. This is possible through the implementation of the so called {\it Euler's principle} \cite{tasche2007}. According to this principle, if $\rho(L)$ expresses the total risk of a loss position $L$, then the {\it individual} risk contribution of the $j$-th component, in our case of the $j$-th sub-domain $\Omega_j$, is calculated according to the definition $\rho(L_j|L) = \frac{d\rho}{dh}(L+h L_j)|_{h=0}.$ Combining the notion of Wasserstein barycentric risk measures and the described risk allocation principle, closed-form solutions are derived for allocating the risk spatially. The results are stated in the next proposition.

\begin{proposition}\label{prop-2}
Given a set of priors models belonging to the Location-Scatter family $\mathcal{M}=\{Q_1,...,Q_N\}$ with $Q_i = LS(m_i, S_i)$, $m_i\in\R^N$, $S_i \in \mathbb{P}(N)$ for $i=1,2,...,N$, the manager's aversion preferences from the prior set $\gamma>0$ and $w\in\Delta^{N-1}$ the degree of realism that is allocated to each provided model, then with respect to the risk mapping \eqref{risk-map} and under the Wasserstein risk measure, the following results are derived:\\

\noindent A. The total harvest risk associated to the optimal control problem \eqref{opt-3} is calculated as
\begin{equation}\label{total-risk}
\rho_W( L ) = G\langle \widetilde{\alpha}, m_B \rangle + \frac{\gamma}{2} G^2 \|\widetilde{\alpha} \|_2^2,
\end{equation}
while the harvest risk allocated to the $j$-th sub-domain of operation $\Omega_j$ for $j=1,2,...,N$ is calculated as 
\begin{equation}\label{indiv-risk}
\RW(L_j | L) = G_j\langle \widetilde{\alpha}, m_B \rangle + \gamma G G_j\| \widetilde{\alpha} \|_2^2
\end{equation}
where $\widetilde{\alpha} := -\left( (e^{T M} - I) M^{-1} \right)^* \alpha$, $m_B := \sum_{i=1}^N w_i m_i$, $G_j := \frac{\theta}{T \Lambda(\alpha)}(B_{D,j}/D_j)^{-1/\beta}$, $G := \sum_{j=1}^N \pi_j G_j$ and $\Lambda(\alpha)$ as defined in \eqref{Lama}.\\

\noindent B. Assuming further that $\theta>0$ and that the terminal time horizon $T>0$ is sufficiently large, the robust optimal harvest policy\footnote{By robust optimal policy we refer to the optimal policy corresponding to the probability measure for the initial condition corresponding to the maximizing measure in the robust representation of the risk measure in \eqref{rm-1}.} related to problem \eqref{opt-3} is obtained in closed form as
\begin{equation}
\widetilde{c}(t) = \frac{\theta}{\Lambda(\alpha)} \widetilde{B}A\widetilde{A}e^{tM} k_0, \,\,\,\, k_0 \sim Q_* = LS(m_B+ \gamma G \widetilde{\alpha}, S_B)
\end{equation}
where, $Q_*$ being the maximizer of problem \eqref{rm-1} with $S_B \in \mathbb{P}(N)$ satisfying the equation $S_B = \sum_{i=1}^N w_i (S_B^{1/2} S_i S_B^{1/2})^{1/2}$.
\end{proposition}

\begin{proof}[Proof of Proposition \ref{prop-2}]
A. The risk mapping which models total potential loss $L$ as expressed in \eqref{risk-map} is of affine form with respect to $k_0$. For a linear risk mapping of the form $\Phi_0(Z) = \langle \beta, Z \rangle$ and a prior set consisting of probability measures belonging to the Location-Scatter family, by applying Proposition 2.10 in \cite{papayiannis2018convex} we get that $\rho_W(L) = \langle \beta, m_B \rangle + \gamma/2 \|\beta\|^2_2$ where $m_B = \sum_i w_i m_i$. Substituting $\beta$ with $G \widetilde{\alpha}$ we get the stated result regarding the total harvest risk. Working with the linear risk mappings $\Phi(k_0) = G\langle \widetilde{\alpha}, k_0 \rangle$ and $\Phi_{0,i} = G_i \langle \widetilde{\alpha}, k_0 \rangle$ related to the firm's losses $L$ and $L_i$ for $i=1,2,...,N$, and combining the calculation result obtained in Proposition \ref{prop-1} and the definition of the Euler's risk allocation, then the stated result concerning the risk allocation is retrieved.\\

\noindent B. From Proposition 2.10 in \cite{papayiannis2018convex} it holds that the maximizer of the Wasserstein barycetric risk measure $Q_*$, for a linear risk mapping $\Phi_0(k_0) = \langle \beta, k_0 \rangle$ and a prior set $\M$ consisting of $N$ Location-Scatter type probability models, is also a member of the Location-Scatter family with location parameters $m_* = m_B + \gamma \beta$ and dispersion matrix satisfying the equation $S_* = S_B = \sum_{i=1}^N w_i (S_B^{1/2} S_i S_B^{1/2})^{1/2}$ the solution of which can be obtained numerically by the fixed-point scheme proposed in \cite{alvarez2016}. Then, the optimizer of problem \ref{opt-3} is obtained from the closed-form solution stated in \eqref{ohr-2} as a consequence of Proposition \ref{prop-1}, where $k_0$ being a random variable the random behaviour of which is described robustly by the probability measure $Q_* = LS(m_*, S_*)$. As a result, the robust optimal controls of the problem \eqref{opt-3} under model uncertainty are characterized by the model $Q_*$ as stated in Proposition \ref{prop-2}. 
\end{proof}

Note that the above results concerning the risk of the loss position depends only on the location characteristics of the aggregate probability measure $m_B \in \R^N$. This happens since the risk mapping is of affine form with respect to $k_0$. Different risk mappings may lead to analytic or semi-analytic formulations for the case of Location-Scatter families (please see Proposition 2.10 in \cite{papayiannis2018convex}). In the case where $\M$ is a singleton, the result in Proposition \ref{prop-2} holds by replacing $m_B$ with $m_0$ being the location parameters of the single prior model $Q_0$. In this case, the parameter $\gamma>0$ is realized just as aversion preferences from the certain provided model $Q_0$ since no model uncertainty exists in this case (there is only one model). 

In order to make clear the uncertainty effects alongside to the manager's preferences, consider the total averaged loss output as defined in \eqref{risk-map} which clearly depends on the states $k_0$. In the simple case where $k_0 \in \R^N$ are known, the output is a number (deterministic case). If a model $Q_0 = LS(m_0, S_0)$ is provided concerning the initial states' random behaviour, and no aversion preferences are provided ($\gamma \to 0$), then the total loss output can be considered as a random variable with distribution $LS(\beta' m_0, \beta' S_0 \beta)$. However, if aversion preferences have been set ($\gamma>0$) the distribution of the loss output is given by the model $LS(\beta'(m_0+\gamma G \widetilde{\alpha}), \beta' S_0 \beta)$, where the level of manager's disbelief concerning the provided model for $k_0$ is quantified by the difference in the location characteristics, i.e. the term $\gamma G \beta' \widetilde{\alpha}$. In the model uncertainty case, the barycenter of the provided models (as determined by the weights $w\in\Delta^{N-1}$), $Q_B = LS(m_B, S_B)$, is playing the role of the model that aggregates the information of the prior set $\M$, and any aversion preference is expressed as a deformation from this model. In particular, for the no aversion preferences from the prior set ($\gamma \to 0$), the random behaviour of the loss output is described by the model $LS(\beta' m_B, \beta' S_B \beta)$, while for any $\gamma >0$ the resulting model for $k_0$ is $LS(\beta' (m_B + \gamma G \widetilde{\alpha}), \beta' S_B \beta)$. In the later case, the difference between the aversion and no aversion case is the location deformation term $\gamma G \beta' \widetilde{\alpha}$, however the location and dispersion characteristics are estimated with respect to the weights $w\in\Delta^{N-1}$ that have been determined by the manager. Therefore, even for the same aversion parameter $\gamma>0$ but with different weight allocations, two different managers could get much different loss distributions (depending of course on the homogeneity of the models in the prior set $\M$).

\section{Conclusions}

An optimal harvesting problem under the framework of model uncertainty for the initial conditions of the underlying dynamics is studied. Employing the class of Fr\'echet risk measures, the model uncertainty is treated robustly through the notion of Wasserstein barycenter, total harvest risk and risk allocation computations are performed and analytic formulas are derived, while simultaneously robust harvesting policies are obtained in analytic form.

\section*{Acknowledgements} The author would like to thank the editor and the associated referees for their insightful comments and suggestions which helped in improving the presentation of this work.


\end{document}